\documentclass[11pt, english, draft]{article}
\usepackage{amssymb,amsmath,amsthm}
\usepackage{graphicx,color}
\title{{\bf Groups that have a Partition by Commuting 
Subsets}}
\author{{\bf T. Foguel}, {\bf J. Hiller}, {\bf Mark L. Lewis} and {\bf A. R. Moghaddamfar}\\[0.3cm]
{\em In memory of Professor Carlo Casolo.}}

\newtheorem{theorem}{Theorem}[section]

\newtheorem{corollary}[theorem]{Corollary}

\newtheorem{proposition}[theorem]{Proposition}
\newtheorem{remark}[theorem]{Remark}
\newtheorem{lm}[theorem]{Lemma}

\newtheorem*{coj*}{Conjecture}
\begin{document}
\newcommand{\f}{\frac}
\newcommand{\sta}{\stackrel}
\maketitle
\begin{abstract}
\noindent  Let $G$ be a nonabelian group.
We say that $G$ has an abelian partition, 
if there exists a partition of $G$ into commuting 
subsets $A_1,  A_2,  \ldots, A_n$ of $G$, 
such that $|A_i|\geqslant 2$ for each $i=1, 2, \ldots, n$. 
This paper investigates problems relating to group
with abelian partitions. Among other results, we show that every finite
group is isomorphic to a subgroup of a group with an abelian partition and
also isomorphic to a subgroup of a group with no abelian partition. We also
find bounds for the minimum number of partitions for several families of
groups which admit abelian partitions -- with exact calculations in some
cases.  Finally, we examine how the size of partitions with the minimum number of  parts behaves with
respect to the direct product.
\end{abstract}


\renewcommand{\baselinestretch}{1}
\def\thefootnote{ \ }
\footnotetext{{\em $2010$ Mathematics Subject Classification}:
Primary: 20D06, Secondary: 05C25.}

\section{Introduction}

In this paper, all groups and graphs are finite and we will consider only simple graphs.  One way to study a group is to attach a graph to the group and use properties of the graph to understand the group. There are a number of different graphs that have been studied in relationship to groups.  We focus on the commuting graph of a group.  Suppose $G$ is a finite group and $X$ is a nonempty  subset of $G$. The commuting graph on the set $X$, which  we denote by ${\cal C}(G,X)$, has $X$ as its vertex set with  $g, h\in X$ joined by an edge whenever they commute as elements  of $G$, i.e.,  $gh=hg$.  We put  $\Delta(G)={\cal C}(G,G)$  and call it the {\em commuting graph of} $G$. We will review the relevant definitions in the text. 

We are especially interested in studying commuting graphs that can be partitioned into disjoint cliques each of which has size at least $2$.  Note that for the commuting graph, a clique in the graph corresponds to a subset of the group which consists of commuting elements.  Given a group $G$, a set $A \subseteq G$ is {\em commuting} if all pairs of elements in the set commute. 

A group $G$ is called an {\em abelian partitionable group} (or an {\em AP-group}) if it can be partitioned into commuting subsets 
$A_1, A_2, \ldots, A_n$ (for some $n$), such that $|A_i|\geqslant 2$ for each $i$. In this situation, $A_1\uplus A_2\uplus \cdots\uplus A_n$ is called an ($n$-){\em abelian partition} of $G$. For notational simplicity, we will always assume that $1\in A_1$.  If $G$ is not an AP-group we say that $G$ is a {\em NAP-group}.  

As far as we can tell, AP-groups were first introduced in \cite{MMS} where the authors classified the AP-groups that have an abelian partition with $n = 2$ and with $n = 3$.  This came out of looking at commuting graphs that satisfied certain numerical conditions that were studied in \cite{MM}.  The study of these groups is also related to the study of groups where the commuting graph and related graphs are split graphs which have been studied in \cite{AM} and \cite{LLMM}.

 In \cite{FH}, it is shown that the dihedral groups of order $2n$, with $n\geqslant 3$ odd, are NAP-groups and that certain direct products of dihedral groups are NAP-groups.  In our paper, we will show that both the class of AP-groups and the class of NAP-groups are ``large'' in the sense that every group is isomorphic to a subgroup of some group in the class. 

\medskip
\noindent{\bf Theorem A.} {\em Let $G$ be a finite group.  Then the following conditions hold:
\begin{itemize}
\item[{\rm (1)}] There is an AP-group $H$ so that $G$ is isomorphic to a subgroup of $H$.
\item[{\rm (2)}]  There is a NAP-group $K$ so that $G$ is isomorphic to a subgroup of $K$.
\end{itemize}}

For the remainder of the paper, we focus on AP-groups.  It is natural to ask if simple groups are necessarily AP-groups.  At this time, we have only been able to show this for $L_2 (q)$ where $q$ is a prime power and ${\rm Sz} (q)$ where $q = 2^{2n+1}$ for a positive integer $n$. 
 
\medskip
\noindent {\bf Theorem B.}
{\em Suppose $G$ is either $L_2 (q)$ where $q$ is a prime power or ${\rm Sz} (q)$ where $q = 2^{2n+1}$ for a positive integer $n$. Then $G$ is an AP-group.}
\medskip

\indent A {\em minimal abelian partition} is one of the least cardinality, and the size of a minimal abelian partition of $G$ is denoted by $\vartheta_a(G)$ (and for convenience we shall write $\vartheta_a(G)=0$ if $G$ is a NAP-group).   We will refer to this number as the {\em minimal AP-degree} of $G$.  We work to obtain several bounds on this degree.  In this direction, the first result we obtain is the following theorem:

\medskip
\noindent {\bf Theorem C.} {\em 
Let $G$ be a nonabelian group, and let $A$ be an abelian subgroup of maximal order in $G$.  Then $\vartheta_a(G) \leqslant [G:Z(G)]-[A:Z(G)]+1$.} 
\medskip

\indent We will also find some examples of groups where this bound is met.  For our next bound, we use $c(G)$ to denote the number of conjugacy classes of $G$.  We then have the bound:

\medskip
\noindent {\bf Theorem D.}  {\em 
Let $G$ be an AP-group.  Then $\vartheta_a(G)\geqslant \lceil |G|/c(G)\rceil$.} 
\medskip

\indent A set $B\subset G$ is {\em noncommuting} if no pair of distinct elements in the set commute. The maximum size of  noncommuting subsets of $G$ is denoted by $n(G)$.  It is actually quite easy to see that $\vartheta_a (G) \geqslant n (G)$.  We find a couple of families of groups where we have the equality $\vartheta_a (G)=n(G)$.

We close this introduction by noting that the idea of looking at partitions of groups is one that has a long history.  Baer, Kegel, and Suzuki classified the finite groups that are partitioned by subgroups (see \cite{Baer1}, \cite{Baer2}, \cite{Baer3}, \cite{kegel} and \cite{Suzuki3}).  Also, Zappa has written an interesting expository article on groups having a partition of subgroups in \cite{zappa}.  While in general, our abelian partition need not consist of subgroups, we will see in some important cases, that an abelian partition will be given by a partition of subgroups.

\section{Graph theory and AP-groups}

All graphs considered here are finite, simple and undirected. 
Let $\Gamma=(V_\Gamma, E_\Gamma)$ be a graph. 
In the case when $E_\Gamma=\emptyset$, we say that 
$\Gamma$ is an {\em edgeless graph} (or {\em  empty graph}). 
A {\em complete set} (or a {\em clique}) in $\Gamma$ is a subset of 
$V_\Gamma$ consists of pairwise adjacent vertices. Similarly, 
an
{\em independent set} in $\Gamma$ is a set 
of pairwise nonadjacent vertices. 

A graph $\Gamma$ is called a {\em split graph} if 
$V_\Gamma$ can be partitioned into an independent set  $I$ and a 
complete set $C$. In this case, $V_\Gamma=I\uplus C$ is called 
a {\em split partition} of $\Gamma$.  Split graphs introduced by F$\rm \ddot{o}$ldes and 
Hammer in \cite{Foldes-Hammer}.  

A generalization of split graphs was introduced 
and investigated under the name $(m, n)$-graphs in \cite{Brandstadt}. 
 A graph $\Gamma$ is a $(m, n)$-graph if its vertex set can be partitioned into 
 $m$  independent sets $I_1, \ldots, I_m$ and $n$ complete sets $C_1, \ldots, C_n$.  
 In this situation,  the partition
 $$V_\Gamma=I_1\uplus I_2\uplus\cdots\uplus I_m\uplus  C_1\uplus C_2\uplus\cdots\uplus C_n,$$
 is called an $(m, n)$-split partition of $\Gamma$. 
 Thus, $(m, n)$-graphs are a natural generalization of split graphs, 
 which are precisely $(1,1)$-graphs.    
 We also note that a $(0,1)$-graph is a complete graph, while a
 $(1,0)$-graph is an edgeless graph.    
 
We now return to the commuting graphs associated with finite groups. 
We denote by $\nabla(G)$ the complement graph of  $\Delta(G)$,  and call it the {\em noncommuting graph of} $G$.
As mentioned in the Introduction, we 
 are interested in studying  those finite groups $G$ for 
 which the commuting graph $\Delta(G)$ is an $(0, n)$-graph 
 with a $(0, n)$-split partition
 $$G=C_1\uplus C_2\uplus\cdots\uplus C_n,$$
 where $|C_i|\geqslant 2$ for each $i$, or equivalently, 
the noncommuting graph $\nabla(G)$ is an $(n, 0)$-graph 
 with a $(n, 0)$-split partition
 $$G=I_1\uplus I_2\uplus\cdots\uplus I_n,$$
 where $|I_i|\geqslant 2$ for each $i$. Clearly, $\Delta(G)$ is an $(0, 1)$-graph
 if and only if $G$ is a nontrivial abelian group. 

 
\section{Groups with abelian partitions}

In this section, we first investigate the existence problem of abelian partitions for finite groups.  Clearly, by the definition, all nontrivial abelian groups are AP-groups,  hence our focus will be on nonabelian groups.  More precisely,  we will show in the next section that $\vartheta_a(G)=1$ if and only if $G$ is abelian. Thus, throughout this section, whenever we refer to a group it should be assumed that it is both {\em nonabelian} 
and {\em finite}, unless otherwise specified.  

It is known that if $A_1\uplus A_2\uplus \cdots\uplus A_n$ is an $n$-abelian partition of a (nonabelian) group $G$, then $n\geqslant 3$ (see \cite[Lemma 1.2]{MMS}).  Recall that we are assuming that each part in our partition has at least two elements, so we have $|G|\geqslant 2n$.  We will see that the nonabelian group of order $6$ is not an AP-group; so we have  $|G|\geqslant 8$.

A group $G$ with {\em nontrivial center} will always be an AP-group; indeed, the cosets of $Z(G)$ are  commuting sets and they partition $G$. 
Additionally, a group $G$ of {\em odd order} will also be an AP-group.  To see this,  it is enough to consider the commuting sets $\{x, x^{-1}\}$ for any nontrivial element $x\in G$ and add the identity element to one such set. Collectively, the preceding discussion shows that the problem of finding an abelian partition for a group essentially reduce to {\em centerless groups of even orders}.  A group $G$ is called centerless if $Z(G)$ is the identity subgroup of $G$.

 We discuss in the sequel some elementary results on $\vartheta_a (G)$ when $G$ is an AP-group. 
Let $G$ be an AP-group with $\vartheta_a(G)=m\geqslant 1$, and we fix for $G$ a minimal abelian partition:
$$G=A_1\uplus A_2\uplus \cdots\uplus A_m.$$ 
An observation that is often useful when working with AP-groups $G$ is that if $A\subseteq G$ is commuting, we have $$A\subseteq \bigcap_{a\in A} C_G(a).$$
Moreover, we can choose an element $a_i\in A_i$ such that $C_G(a_i)$  has minimum possible order; we see that 
$$G=\bigcup_{i=1}^m C_G(a_i),$$
is a covering of $G$ by $m$ centralizers, and we obtain $|G|\leqslant \sum_{i=1}^m |C_G(a_i)|$.  We now obtain another covering of $G$ by centralizers when we have commuting set whose size is strictly larger than the size of a minimal abelian partition for the group. 

\begin{lm}\label{union} 
If $\{a_1, a_2, \ldots, a_k\}$ is a commuting set of a group $G$ with $k>  \vartheta_a(G)$, then we have 
$$G= \bigcup_{1\leqslant i<j\leqslant k}C_G(a_{i}^{-1} a_j).$$
\end{lm}

\begin{proof}  
Assume the contrary and choose $g\in G\setminus \bigcup_{1\leqslant i<j\leqslant k}C_G(a_{i}^{-1} a_j)$. 
We claim that  $\{ga_1, ga_2, \ldots, ga_k\}$ is a noncommuting set of $G$. If not, then there is a pair $i<j$ such that
$ga_jga_i=ga_iga_j$, or equivalently, $(a_i^{-1}a_j)g=g(a_ja_i^{-1})=g(a_i^{-1}a_j)$,  which shows that $g\in C_G(a_i^{-1}a_j)$, a contradiction. We now have $k\leqslant n(G)\leqslant \vartheta_a(G)$. This contradiction completes the proof.
\end{proof} 

Referring to the already mentioned fact that any group $G$ with nontrivial center $Z(G)$ can be partitioned into the cosets of $Z(G)$ as 
commuting subsets of $G$, we see that $\vartheta_a(G)\leqslant [G:Z(G)]$. However, Theorem C gives an improved upper bound on $\vartheta_a(G)$, which we prove now. 


\begin{proof}[Proof of Theorem C] 
Let $A$ be an abelian subgroup of $G$ of maximal order. It is
obvious that $Z(G)\leqslant A$. Suppose that $[A:Z(G)]=k$ and $T_A=\{a_1,
a_2,  \ldots, a_k\}$ be a transversal of $Z(G)$ in $A$ with $a_1=1$. We
assume that $[G:Z(G)]=n$ and let $T_G:=T_A\cup \{b_1, b_2, \ldots, b_{n-k}\}$ be a
transversal of $Z(G)$ in $G$. It now follows that
$$G=A\uplus Z(G)b_1\uplus  Z(G)b_2\uplus \cdots \uplus  Z(G)b_{n-k},$$
is an abelian partition for $G$, from which the desired conclusion
$$\vartheta_a(G)\leqslant n-k+1=[G:Z(G)]-[A:Z(G)]+1,$$ will follow.
\end{proof}

We provide some examples that show that the upper bound in Theorem C is sharp.  We will make use of the following lemma in our examples.

\begin{lm} \label{minimal}
Let $G$ be a group and suppose there exist $a_1, \dots, a_n \in G$ so that $C_G (a_i)$ is abelian for $i = 1, \dots, n$.  If $G = A_1 \uplus A_2 \uplus \dots \uplus A_n$ is an abelian partition of $G$ with $A_1 = C_G (a_1)$ and $A_i = C_G (a_i) \setminus Z(G)$ for $i = 2, \dots, n$, then $A_1 \uplus A_2 \uplus \dots \uplus A_n$ is a minimal abelian partition of $G$.
\end{lm}

\begin{proof}
Suppose $G = B_1\uplus B_2 \uplus \cdots \uplus B_m$ is another abelian partition of $G$.  For each $i = 1, \dots, n$, we must have $a_i$ in $B_j$ for some $j$.  Observe that $B_j \subseteq C_G (a_i) = A_i \cup Z(G)$.  Since the $A_i$'s are pairwise disjoint, this implies that no $B_j$ contains more than one $a_i$.  We thus conclude that $m\geqslant n$.  This implies that $A_1 \uplus A_2 \uplus \dots \uplus A_n$  is a minimal abelian partition of $G$.
\end{proof}

We now present examples that meet the upper bound of Theorem C.  Suppose $G$ is one of the following groups of order $4n$:
$$D_{4n}=\langle a, b \ \mid \ a^{2}=1,  b^{2n}=1, a^{-1}ba=b^{-1}\rangle,$$ or 
$$Q_{4n}=\langle a, b \ \mid \  b^{2n}=1, b^n=a^2,
a^{-1}ba=b^{-1} \rangle,$$ where $n\geqslant 2$.
Here, $Z(G)=\{1, b^n\}$ and $A=\langle b\rangle$. Hence,
we have $$\vartheta_a(G)\leqslant [G:Z(G)]-[A:Z(G)]+1=2n-n+1=n+1.$$
Note, however, that we have a abelian partition of $G$ as follows. Take 
$$A_1=\langle b\rangle, \ {\rm and} \ A_i=\{ab^{i-2}, ab^{n+i-2}\}, \ i=2, 3, \ldots, n+1.$$
Then certainly $G=A_1\uplus A_2\uplus \cdots\uplus A_{n+1}$ is an abelian partition of $G$.  To see that it is minimal, observe that $C_G (b) = A = A_1$ and $C_G (ab^{i-2} )= \{ 1, b^{n}, ab^{i-2},ab^{n + i-2} \} = A_i \cup Z (G)$ for $i = 2, 3, \ldots, n+1$.   
By Lemma \ref{minimal}, $A_1\uplus A_2\uplus \cdots\uplus A_{n+1}$ is a minimal abelian partition, and so we have $\vartheta_a(G)=n+1$.   (This is a special case of Corollary \ref{Dihedral-GeneralizedQuaternion}.) 

It would be an interesting question to see if one could classify all groups where equality holds in Theorem C.
We next consider minimal abelian partitions of direct products.

\begin{lm}\label{direct-product}  
Let $G = H\times K$ be a direct product of two groups $H$ and $K$ for which $\vartheta_a (H)$ and $\vartheta_a (K)$ both are not $0$. Then $\vartheta_a (G) \leqslant  \vartheta_a(H) \times \vartheta_a(K)$.
\end{lm} 

\begin{proof}
The assertion is immediate, since if $H=A_1\uplus A_2 \uplus \cdots \uplus A_m$ is a minimal 
 abelian partition of $H$, and $K=B_1\uplus B_2 \uplus \cdots \uplus B_n$ is a minimal abelian partition of $K$, then  the subsets 
$A_i\times B_j$,  $1 \leqslant i \leqslant m$, $1 \leqslant j \leqslant n$ form an abelian partition of 
$G$, where $A_i\times B_j = \{(a, b) \mid a \in A_i, \ b\in B_j \}$. 
\end{proof}

\begin{remark} {\rm Note that the result of Lemma \ref{direct-product} is not true without the assumption 
$\vartheta_a (H)\neq 0\neq \vartheta_a (K)$. For instance, if $H$ is a NAP-group
and $K$ is a nontrivial abelian group, or equivalently $\vartheta_a (H)=0$
and  $\vartheta_a (K)=1$, then  $G=H \times K$ will be an AP-group.  Obviously, in this case $$\vartheta_a (G) > 0 =  \vartheta_a(H) \times \vartheta_a(K). $$  

We also do not know of any examples of groups $H$ and $K$ where the inequality in Lemma \ref{direct-product} is strict.  We ask whether it is possible to prove that in fact we must have equality in Lemma \ref{direct-product}.}
\end{remark}

We next show that in an a minimal abelian partition that each part must contain a noncentral element.

\begin{lm}\label{centralelements}   Let $G$ be an AP-group with a minimal abelian partition $A_1\uplus A_2\uplus \cdots\uplus A_m$, (recall $m\geqslant 3$). Then for each $i$, the set $A_i$ contains a noncentral element of $G$. In other words, for each $i$, $A_i \nsubseteq Z(G)$. 
\end{lm}

\begin{proof}  If $A_i \subseteq Z(G)$ for some $i$, then  for each $j\neq i$, $A_i\cup A_j$ is a commuting set of $G$.  We obtain a smaller abelian partition of $G$ by replacing $A_i$ and $A_j$ with $A_i \cup A_j$.  This contradicts the minimality of $m$. 
\end{proof}

We now show that we can always arrange for a minimal abelian partition so that the center of the group is contained in the first part of the partition.

\begin{lm}\label{firstsummand}   Let $G$ be a (nonabelian) AP-group with $\vartheta_a(G)=m (\geqslant 3)$. Then there exists a minimal abelian partition 
$A_1\uplus A_2\uplus \cdots\uplus A_m$ of $G$ such that $Z(G) < A_1$.
\end{lm}

\begin{proof}  Let $B_1\uplus B_2\uplus \cdots\uplus B_m$ be a minimal abelian partition of $G$.
If $|Z(G)|=1$, then $Z(G)\subset B_1$ and there is nothing to prove, so assume $|Z(G)|>1$. Let
 $z\neq 1$ be a central element of $G$. Then $z\in B_i$ for some $i$. We claim that $|B_i|\geqslant 3$.
Otherwise, $|B_i|=2$, say 
$B_i=\{z, y\}$. Now, we consider the element $yz$ and assume that $yz\in B_k$ for some $k$.  But then 
$B_i\cup B_k$ forms a commuting set of $G$, which contradicts the minimality of $m$.  
This contradiction shows $|B_i|\geqslant 3$ as claimed.
Now, we take 
$B_1'=B_1\cup \{z\}$ and $B_i'=B_i\setminus \{z\}$. Therefore, we see that 
$$G=B_1'\uplus B_2\uplus \cdots\uplus  B_{i-1} \uplus B_i'\uplus B_{i+1}\uplus \cdots \uplus B_m,$$
is a minimal abelian partition of $G$ such that $z\in B_1'$.  The above process of adding a central element to $B_1$ may therefore repeated. 
At the end, after at most $(|Z(G)|-1)$ steps, we obtain the required minimal abelian partition $A_1\uplus A_2\uplus \cdots\uplus A_m$. Note that by Lemma \ref{centralelements}  it follows that $Z(G)\subset A_1$.  By Lemma \ref{centralelements}, we see that $Z (G) \ne A_1$, so we must have $Z(G) < A_1$.
\end{proof}

The ordered pair $(x, y)\in G\times G$ is called a {\em commuting pair} if  $xy=yx$. In \cite{Erdos-Straus} Erd\"{o}s and Tur\'{a}n proved the following result.

\begin{lm} {\bf (Erd\"{o}s -Tur\'{a}n)}\label{ET1967} The number of commuting pairs of elements of a group $G$ is $|G|c(G)$
where $c(G)$ is the number of conjugacy classes of $G$.
\end{lm}

With this Lemma in hand, we prove Theorem D.

\begin{proof}[Proof of Theorem D]  Suppose $G=A_1\uplus A_2\uplus \cdots\uplus A_m$ is an abelian partition of $G$.
We make use of the following inequality
$$ \sum_{i=1}^m |A_i|^2\geqslant \frac{1}{m} \left(\sum_{i=1}^m |A_i|\right)^2,$$
which follows at once from the well-known Cauchy-Schwarz inequality$^1$\footnote{$^1$Let $a_1, a_2, \ldots, a_m$ and  $b_1, b_2, \ldots, b_m$ be arbitrary real numbers. The well-known Cauchy-Schwarz's  inequality asserts that $(\sum_{i=1}^m a_ib_i)^2\leqslant (\sum_{i=1}^m a_i^2)(\sum_{i=1}^m b_i^2)$.}\!\!\!\!. Now, by using the fact that $\sum_{i=1}^m |A_i|=|G|$, we obtain 
$$\sum_{i=1}^m |A_i|^2\geqslant \frac{|G|^2}{m}.$$ On the other hand, since $A_i$ is a commuting subset of $G$ for each $i$, $1\leqslant i\leqslant m$, the set $A_i^2=A_i\times A_i$  consists of commuting pairs of elements of $A_i$, and  $\cup_{i=1}^m A_i^2$ is a disjoint union. It follows by Lemma \ref{ET1967} that 
$$|G|c(G)\geqslant \sum_{i=1}^m |A_i|^2.$$
Collecting the above two inequalities together, we obtain $m\geqslant \lceil |G|/c(G)\rceil$.
\end{proof}

As a corollary to Theorem D, we prove that $c (G)$ is an upper bound for the minimal size of the $A_i$'s in an abelian partition of $G$. 

\begin{corollary}\label{inverse}   Let $G$ be an AP-group and let 
\begin{equation}\label{m1} 
G=A_1\uplus A_2\uplus \cdots\uplus A_m,
\end{equation} be an abelian partition of $G$. Then, the following condition holds:
Let $s$ denote the minimal size of the $A_i$'s. Then, we have  
$s\leqslant c(G)$.
\end{corollary}

\begin{proof}  As in the proof of Theorem D,  we consider the inequality $|G|c(G)\geqslant \sum_{i=1}^m |A_i|^2$. Now, if $s$ denotes the minimal size of the commuting sets $A_i$, then we obtain 
$$|G|c(G)\geqslant \sum_{i=1}^m |A_i|^2\geqslant \sum_{i=1}^m s^2=ms^2\geqslant \frac{|G|s^2}{c(G)},$$
where we have used Theorem D to obtain the last equality. The result is now immediate.
\end{proof}

We do not know of any examples of nonabelian groups where the bound in Lemma \ref{inverse} is obtained.  It would be interesting to see if one could prove a better general bound.


\section{Computing the minimal AP-degree $\vartheta_a (G)$}

We now focus specifically on computing $\vartheta_a (G)$ for various groups $G$.  One question that we would like to answer is: what integers can occur as $\vartheta_a (G)$.  As a direct consequence of the definition, we have the following.

\begin{lm}\label{lem2.2}  Let $G$ be a nontrivial group.  Then $\vartheta_a(G)=1$ if and only if $G$ is abelian.
\end{lm}
 
On the other hand, it follows immediately using the pigeon-hole principle that $n(G)\leqslant \vartheta_a(G)$ for every finite group $G$.
Moreover, if $x$ and $y$ don not commute, then $\{x, y, xy\}$ is a noncommuting subset of $G$, this forces $n(G)\geqslant 3$. 
Thus for any nonabelian
AP-group $G$, $\vartheta_a(G) \geqslant 3$. So we have the following result.

\begin{lm}\label{lem2.3}  If $G$ is a nonabelian group, then $\vartheta_a(G)\geqslant  n(G)\geqslant 3$.  In particular, there is no finite group $G$ with $\vartheta_a(G) = 2$.
\end{lm}

 In the case that $G$ is a nonabelian group of order $8$, it can be shown that  $\vartheta_a(G)=3$, so the
bound  can be achieved. In general, it was shown in \cite{MMS} that $\vartheta_a(G)=3$ if and only if  $G$ is isomorphic to the direct product $P\times Q$, where $P$ is a $2$-group with $P/Z(P)\cong {\Bbb Z}_2\times {\Bbb Z}_2$  and  $Q$ is an abelian group.   We write ${\Bbb Z}_m$ for a cyclic group of order $m$.  In the same paper, it is proved that $\vartheta_a(G)=4$ if and only if  $|Z(G)|\geqslant 2$ and $G/Z(G)$ is isomorphic to  ${\Bbb Z}_3\times {\Bbb Z}_3$ or ${\Bbb S}_3$.

It might be of interest to investigate $\vartheta_a(G)$ for families of finite groups, particularly finite simple groups. Let us consider the problem of finding  $\vartheta_a(G)$ for  certain finite groups.


\subsection{AC-groups} 

A group $G$ is called an {\em AC-group}, if the centralizer of every noncentral element of $G$ is abelian.  These groups have been investigated by many authors in various contexts, see for instance \cite{DHJ, Rocke, Schmidt}$^1$\footnote{$^1$Note that some of these references refer to these groups as CA-groups, and others use CA-groups to denote groups where every nonidentity element has an abelian centralizer.}\!\!\!\!.  A routine argument shows that $G$ is an AC-group if and only if for all noncentral elements $x$ and $y$ in $G$ either $C_G(x) = C_G(y)$ or $C_G(x)\cap C_G(y) = Z(G)$. We will use it afterward without mention. 

As mentioned previously,  $n(G)\leqslant \vartheta_a(G)$ since two elements that do not commute cannot be in the same commuting set. The following result  determines a class of groups $G$ with $\vartheta_a(G)=n(G)$.   This next Proposition can be viewed as a sort of converse for Lemma \ref{minimal} when $G$ is an AC-group.

\begin{proposition}\label{AC-groups}  Let $G$ be an AC-group where all the centralizers have size at least $3$.  Then, we have   $\vartheta_a(G)=n(G)$. 
\end{proposition}

\begin{proof}
Suppose that $G$ is an AC-group where all the centralizers have size at least $3$.  Let $n=n(G)$ and assume that 
$\{x_1, x_2, \ldots, x_n\}$ is a noncommuting subset of maximal cardinality in $G$.  Then, each element $g$ in $G$
must commute with at least one of the $x_i$, and hence $G$ is covered by the centralizers $C_G(x_i)$, $i=1, 2, \ldots, n$,
which intersect pairwise only in $Z(G)$:
$$G=C_G(x_1)\cup C_G(x_2)\cup \cdots \cup C_G(x_n).$$
Put $A_1=C_G (x_1)$ and $A_i=C_G(x_i)\setminus Z(G)$ for $i = 
2, 3, \ldots, n$.  It is not difficult to see that $$G=A_1\uplus A_2\uplus \cdots \uplus A_n.$$  Obviously, $|A_1| = |C_G (x_1)| \geqslant 3$.  Suppose that $i \geqslant 2$.  Observe that $A_i$ will be a union of cosets of $Z(G)$.  If $|Z (G)| \geqslant 2$, then $|A_i| \geqslant |Z(G)| \geqslant 2$.  Thus, we may assume that $Z(G) = 1$.  In this case, $|A_i| = |C_G(x_i) \setminus \{ 1 \}| = |C_G (x_i)| - 1 \geqslant 3 -1 =2$.  Thus, we have $|A_i| \geqslant 2$ in all cases.  Therefore,  $$A_1\uplus A_2\uplus \cdots \uplus A_n,$$ will be an abelian partition of $G$.  (Note that it will be true exactly when all of the 
centralizers have size at least $3$.)  Hence, $v_a (G)\leqslant n$. But, it is always true that $n\leqslant \vartheta_a(G)$, 
and we thus have equality. \end{proof}

We note that the hypothesis in Proposition \ref{AC-groups} that the centralizers have size at least $3$ is necessary.  The group ${\Bbb S}_3$ is easily seen to be an AC-group, but it has centralizers of size $2$ and it is not difficult to see that it is not an AP group.  From Proposition \ref{AC-groups} we obtain the following corollaries.  In the first case, we have $G/Z(G)$ is elementary abelian group of rank $2$.

\begin{corollary} Let $G$ be a $p$-group for some odd prime $p$. If $G/Z(G)$ is elementary abelian of order $p^2$, then we have $\vartheta_a(G)=n(G)=p+1$. 
\end{corollary}

We also obtain the result for dihedral and generalized quaternion groups.

\begin{corollary}\label{Dihedral-GeneralizedQuaternion}  If $G$ is isomorphic to one of the groups $D_{4n}$ or $Q_{4n}$, where 
$n\geqslant 2$ is an integer, then we have 
$\vartheta_a(G)=n(G)=n+1$.
\end{corollary}

We have found one family of groups where the equality $n (G) = \vartheta_a (G)$ holds.  We note that this equality holds in ${\Bbb S}_4$ which is not an AC-group, so there exist other groups that satisfy this equality.  {\em Can we classify all groups that satisfy this equality?}


\subsection{Frobenius groups}

A nontrivial proper subgroup $H$ of a group $G$ is a {\em Frobenius complement} in $G$ if $H\cap H^g =1$ for all $g\in G\setminus H$.  A group which possesses a Frobenius complement is called a {\em Frobenius group}.
A classical result of Frobenius (see \cite[Theorem 7.5]{Gor}) asserts that for a Frobenius complement $H$ in $G$, the subset 
$$ N:=\left(G\setminus \bigcup_{g\in G} H^g\right)\cup  \{1\},$$
(which is uniquely determined by $H$) is a characteristic subgroup of $G$. The subgroup $N$
is said to be the {\em Frobenius kernel} of $G$ (with respect to the Frobenius complement $H$). 

We mention without proof that in fact a Frobenius group has a unique conjugacy class of complements and a unique kernel. Furthermore, one verifies directly from the definition that 
$$G=\bigcup_{n\in N} H^n\cup N. $$
Now, using the following well known facts (see \cite[p. 121]{Isaacs}):
\begin{itemize}
\item[{\rm (a)}] $C_G(n)\subseteq N$ for all $1\neq n\in N$;
\item[{\rm (b)}] $C_H(n)=1$ for all $1\neq n\in N$; and
\item[{\rm (c)}] $C_G(h)\subseteq H$ for all $1\neq h\in H$;
\end{itemize}
it follows at once that if $H$ and $N$ are abelian and $|H|\geqslant 3$, then  
$\vartheta_a (G)=|N|+1$. We now consider the general case where $H$ and $N$ may be nonabelian.

\begin{proposition} Let $G$ be a Frobenius group with Frobenius complement $H$ and Frobenius kernel $N$. If  $|H| \geqslant 3$, then we have
$$\vartheta_a(G)=|N|\vartheta_a(H)+\vartheta_a(N).$$
\end{proposition}

\begin{proof} 
Let $G$ be a Frobenius group with Frobenius kernel $N$ and Frobenius complement $H$.  By Thompson's theorem, we know that $N$ is nilpotent, so $Z(N) > 1$, and we have seen that this implies that $N$ is an AP-group.  It is known that every Frobenius complement has a nontrivial center (see Satz V.8.18 (c) of \cite{Hup}), so $H$ is an AP-group.  If $H$ is abelian, then $A_1 (H) = H$ and $|A_1 (H)| = |H| \geqslant 3$.  When $H$ is nonabelian, we may use Lemma \ref{firstsummand} to assume that $Z(H) < A_1 (H)$ and since $|Z(H)| \geqslant 2$, we have $|H| \geqslant 3$..

Then, we have the following decomposition of $G$: 
$$G=\bigcup_{n\in N} H^n\cup N=H^{n_1}\cup H^{n_2}\cup \cdots \cup H^{n_{t}} \cup N, $$
where $n_1=1$ and $t=|N|$. 
Suppose that $$H=A_1(H)\uplus A_2(H)\uplus \cdots \uplus A_r(H),$$
and 
$$N=A_1(N)\uplus A_2(N)\uplus \cdots \uplus A_s(N),$$
are minimal abelian partitions of $H$ and $N$. Note that $1\in A_1(H)\cap A_1(N)$.

We have seen that $|A_1(H)| \geqslant 3$. Put $A_1^\ast(H)=A_1(H)\setminus \{1\}$.  It is then easy to see that
$$G=\bigcup_{i=1}^{t} \big(A_1^\ast(H)^{n_i}\uplus A_2(H)^{n_i}\uplus \cdots \uplus A_r(H)^{n_i}\big)\uplus A_1(N)\uplus A_2(N)\uplus \cdots \uplus A_s(N), $$
is a minimal abelian partition for $G$, and so  $\vartheta_a(G) \geqslant |N|r+s=|N|\vartheta_a(H)+\vartheta_a(N)$.

To see that we obtain an equality, suppose that $B_1 \uplus \cdots \uplus B_m$ is a minimal abelian partition of $G$.  For each $i$, consider $b_i \in B_i$.  (For $i = 1$, we assume $b_1 \neq 1$.)  We know $b_i$ lies in one of $N$ or $H^n$ for some $n \in N$.  If $b_i \in N$, then $B_i \subseteq C_G (b_i) \leqslant N$ and if $b_i \in H^n$, then $B_i \subseteq C_G (b_i) \leqslant H^n$.  It is not difficult to see that we can use the $B_i$'s to obtain abelian partitions of $N$ and of each of the $H^n$'s.  It follows that $m \geqslant |N|r + s$, and this gives the desired equality.
\end{proof}


\section{Simple Groups}

In the next two subsections, we investigate Theorem B. Before proceeding with the study of simple groups $L_2(q)$ and ${\rm Sz}(q)$, we need some additional definitions and notation. The {\em spectrum} $\omega(G)$ of a finite group $G$ is the set of its element orders,
 and the {\em prime spectrum} $\pi (G)$ is the set of all prime divisors of its order.   The set $\omega(G)$ is closed under divisibility and hence is uniquely determined by the set $\mu(G)$ of elements in $\omega(G)$ which are maximal under the divisibility relation. 
 
\subsection {The projective special linear groups $L_2(q)$} 

In this subsection, we assume that $G$ is a group isomorphic to $L_2(q)$, where $q=p^n$, with $p$ a prime and $n$ a positive integer. We are going to show that $G$ is an AP-group and find an explicit formula for $\vartheta_a(G)$. 
We make a few observations before going on to prove anything (see \cite{Dickson} and \cite[Ch. II, \S 8]{Hup}):

\begin{itemize}
\item[{\rm (a)}]   $|G|=q(q-1)(q+1)/d$ and $\mu(G)=\{p, (q-1)/d, (q+1)/d\}$, where $d={\rm gcd}(2, q-1)$.

\item[{\rm (b)}] Let $P$ be a Sylow $p$--subgroup of $G$. Then  $P$ is an elementary abelian $p$--group of order
$q$, which is a TI--subgroup, and  $|N_G(P)|=q(q-1)/d$.

\item[{\rm (c)}] Let $A\subset G$ be a cyclic subgroup of order $(q-1)/d$. Then $A$ is a TI--subgroup and
the normalizer $N_G(A)$ is a dihedral group of order $2(q-1)/d$.

\item[{\rm (d)}] Let $B\subset G$ be a cyclic subgroup of order $(q+1)/d$. Then  $B$ is a TI--subgroup
and the normalizer $N_G(B)$ is a dihedral group of order $2(q+1)/d$.
\end{itemize}

We recall that a subgroup  $H\leqslant G$  is a {\em TI--subgroup} (trivial intersection subgroup) if for every $g\in G$, either $H^g=H$ or $H\cap H^g=1$. 

Let us first handle the case where $L_2 (4)\cong L_2 (5)\cong {\Bbb A}_5$.

\begin{lm}\label{lem2.4}  $\vartheta_a (L_2(4))=\vartheta_a (L_2(5))=\vartheta_a({\Bbb A}_5)=21$.
\end{lm}

\begin{proof}  Let  $G=L_2(4)\cong L_2(5)\cong {\Bbb A}_5$. Observe that $G$ is an AC-group.  It is easy to check that  $G$ contains 1 element of order 1, 15 elements of order 2, 20 elements of order 3, and 24 elements of order 5, and so $\omega(G)=\{1, 2, 3, 5\}$. In particular, the centralizer of every element of order $p$ in $G$ is a Sylow $p$-subgroup, where $p\in \pi(G)$.  Notice that the Sylow subgroups are abelian and have order at least $3$.  Thus, we may use Proposition \ref{AC-groups}, and it suffices to find a maximal noncommuting set.  

Since all of the pairs of Sylow subgroups intersect trivially and are abelian, we see that a maximal noncommuting set will contain one element of each Sylow subgroup. Hence, it will contain six  elements of order $5$. Similarly, it will contain ten elements of order $3$. Finally, it will contain five elements of order $2$. We conclude that a maximal noncommuting set has $6 + 10 + 5 = 21$ elements.  Thus $\vartheta_a(G) = 21$, and the proof is complete.  \end{proof}

Now we handle $L_2(q)$ when $q > 5$ is a prime power.

\begin{lm}\label{lem2.5}   Let $q$ be a prime power. If $q>5$, then $\vartheta_a (L_2(q))=q^2+ q+1$. Furthermore, we have $\vartheta_a (L_2(2))=0$,  $\vartheta_a (L_2(3))=5$ and $\vartheta_a (L_2(4))=\vartheta_a (L_2(5))=21$.
\end{lm}

\begin{proof}  Let $G=L_2(q)$. Since $L_2(2)\cong {\Bbb S}_3\cong {\Bbb Z}_3\rtimes {\Bbb Z}_2$ and $L_2(3)\cong {\Bbb A}_4\cong ({\Bbb Z}_2\times{\Bbb Z}_2)\rtimes {\Bbb Z}_3$, it is easy to see that  $\vartheta_a (L_2(2))=0$ and $\vartheta_a (L_2(3))=5$.  Moreover, since $L_2(4)\cong L_2(5)\cong {\Bbb A}_5$,  it follows from Lemma \ref{lem2.4} that  
$\vartheta_a (L_2(4))=\vartheta_a (L_2(5))=21$. Hence we may assume that $q>5$. 

We note that if $q$ is a power of $2$, then $G$ will be an AC-group and we could appeal to Proposition \ref{AC-groups}.  When $q$ is odd, this is not true, and we are to going to present a unified argument that works for all $q$.  As already mentioned,  $G$ contains abelian subgroups
 $P$, $A$ and $B$, of orders $q$, $(q-1)/d$ and $(q+1)/d$, respectively, every distinct pair of their conjugates 
 intersects trivially, and every element
of $G$ is a conjugate of an element in $P\cup A\cup B$. 
Let $$G=N_Pu_1\cup \cdots \cup N_Pu_r=N_Av_1\cup \cdots \cup N_Av_s= N_Bw_1\cup \cdots \cup N_Bw_t,$$
be coset decompositions of $G$ by $N_P=N_G(P)$, $N_A=N_G(A)$ and $N_B=N_G(B)$, where
  $r=[G:N_P]=q+1$, $s=[G:N_A]=q(q+1)/2$  and  $t=[G:N_B]=(q-1)q/2$.
Then, we have
$$G=P^{u_1}  \cup \cdots \cup P^{u_r}\cup A^{v_1}\cup \cdots \cup A^{v_s}\cup B^{w_1} \cup \cdots \cup B^{w_t}. $$
Put 
$$P_i=P^{u_i}\setminus \{1\} \ (1\leqslant i\leqslant r), \ \ A_j=A^{v_j}\setminus \{1\} \ (1\leqslant j\leqslant s), \ \   B_k=B^{w_k}\setminus \{1\} \ (1\leqslant k\leqslant t). $$
Then, since 
$$\begin{array}{lll} 
|P_i|=q-1>1, & &  i=1, 2, \ldots, r,\\[0.2cm]
|A_j|=(q-1)/d-1\geqslant (q-1)/2-1>1, & &  j=1, 2, \ldots, s,\\[0.2cm]
|B_k|=(q+1)/d-1\geqslant (q+1)/2-1> 1, & &  k=1, 2, \ldots, t,
\end{array}$$
and $\mu(G)=\{p, (q-1)/d, (q+1)/d\}$, we see that $$G\setminus \{1\}=P_1  \cup \cdots \cup P_r\cup A_1\cup \cdots \cup A_s\cup B_1 \cup \cdots \cup B_t,$$
leads to an abelian partition of $G$. 

For $i = 1, \dots, r$, pick $x_i \in P_i $, for $j = 1, \dots, s$, set $A^{v_j} = \langle a_j \rangle$, and for $k = 1, \dots, t$, set $B^{w_k} = \langle b_k \rangle$.  Note that $C_G(x_i)=P_i \cup \{1 \}$ for $i = 1, \dots, r$; $C_G (a_j)=A_j \cup \{1 \}$ for $j= 1, \dots, s$, and $C_G (b_k) = B_k \cup \{ 1 \}$.  We now may apply Lemma \ref{minimal} to see that  $$P_1  \cup \cdots \cup P_r\cup A_1\cup \cdots \cup A_s\cup B_1 \cup \cdots \cup B_t$$ is a minimal abelian partition of $G$.  
Therefore, we get
$$\vartheta_a (G)=r+s+t=q+1+q(q+1)/2+(q-1)q/2=q^2+q+1.$$
This completes the proof. \end{proof}


\subsection{The Suzuki simple groups ${\rm Sz}(q)$}
Assume now that $G$ is a group isomorphic to Suzuki group ${\rm Sz}(q)$, where $q=2^{2n+1}\geqslant 8$. We start by recalling some well known facts about the simple group $G$ (see \cite{Suzuki2, Suzuki}):
\begin{itemize}
\item[{\rm (a)}]  $|G|= q^2(q-1)(q^2+1)= q^2(q-1)(q-r+1)(q+r+1)$, where $r=2^{n+1}$. Note that  these factors are mutually coprime and play an important role in the structure of the subgroups of $G$.
Moreover, we have $\mu(G)=\{4, q-1, q-r+1, q+r+1\}$.

\item[{\rm (b)}] Let $P$ be a Sylow $2$-subgroup of $G$.
Then  $P$ is a $2$-group of order
$q^2$ with ${\rm exp}(P)=4$, which is a TI-subgroup, and  $N_G(P)$ is a Frobenius groups 
of order $q^2(q-1)$.

\item[{\rm (c)}] Let $A\subset G$ be a cyclic subgroup of order $q-1$. Then $A$ is a TI-subgroup and
the normalizer $N_G(A)$ is a dihedral group of order $2(q-1)$.

\item[{\rm (d)}] Let $B\subset G$ be a cyclic subgroup of order $q-r+1$. Then $B$ is a TI-subgroup and
the normalizer $N_G(B)$  has order $4(q-r+1)$.

\item[{\rm (e)}] Let $C\subset G$ be a cyclic subgroup of order $q+r+1$. Then $C$ is a TI-subgroup and
the normalizer $N_G(C)$ has order $4(q+r+1)$.
\end{itemize}

We now show that $G={\rm Sz} (q)$ is an AP-group and we compute $\vartheta_a (G)$.

\begin{lm} \label{lem2.6}  We have  $\vartheta_a ({\rm Sz}(q))=q^4+q^3-q^2+q-1$, where $q=2^{2n+1}\geqslant 8$.
 
\end{lm}

\begin{proof} Let $G={\rm Sz}(q)$, where $q=2^{2n+1}\geqslant 8$. We apply (a)-(e) using the notation given there. As was mentioned before,  $G$ contains a Sylow 2-subgroup $P$ of order $q^2$ and cyclic subgroups $A$, $B$, and $C$, of orders
$q-1$, $q-r+1$ and  $q+r+1$, respectively.
Moreover, every two  distinct conjugates of them intersect trivially and every element
of $G$ is a conjugate of an element in $P\cup A\cup B\cup C$.

First we consider a Sylow $2$-subgroup $P$ of $G$.  By Theorem VIII.7.9 of \cite{hupII} and Lemma XI.11.2 of \cite{hupIII}, $Z(P)$ is an elementary abelian $2$-group of order $q$ and every element outside $Z(P)$ has order $4$.  Observe that $P$ is the centralizer in $G$ of all of the nontrivial elements of $Z(P)$.  If $x \in P \setminus Z(P)$, then $\langle Z(P), x \rangle \leqslant C_G (x)$.  In the proof of Lemma XI.11.7 of \cite{hupIII}, we see that the elements of order $4$ in $G$ lie in two conjugacy classes. This implies that $|C_G (x)| = 2 |Z(P)|$, from which we deduce that $C_G (x) = \langle Z(P), x \rangle$. In particular, $C_G(x)$ is abelian.
Therefore,
for all $x, y\in P\setminus Z(P)$ either $C_G(x) = C_G(y)$ or $C_G(x)\cap C_G(y) = Z(P)$.
Hence, $Z(P)\cup \{C_G(x)\setminus Z(P) \mid x\in P\setminus Z(P)\}$ forms an abelian partition of $P$. 
Computation yields $\vartheta_a (P)=q-1$.

Next, by the proof of Lemma XI.11.6 in \cite{hupIII}, we see that the cyclic subgroups $A$, $B$, and $C$,  are the centralizers of their nonidentity elements.
Let
$$G = N_Px_1\cup \cdots \cup N_Px_p
=  N_{A}y_1\cup \cdots \cup N_{A}y_a
=  N_{B}z_1\cup \cdots \cup N_{B}z_b
 =  N_{C}t_1\cup \cdots \cup N_{C}t_c,$$
be coset decompositions of $G$ by  $N_P=N_G(P)$, $N_A=N_G(A)$, $N_B=N_G(B)$ and $N_C=N_G(C)$, where

$$\begin{array}{lll}
p & = & [G:N_P]=q^2+1,\\[0.2cm]
a & = & [G:N_A]=q^2(q^2+1)/2,\\[0.2cm] 
b& = & [G:N_B]=q^2(q-1)(q+r+1)/4, \  \mbox{and} \\[0.2cm] 
c&=&[G:N_C]=q^2(q-1)(q-r+1)/4.\end{array}$$
Then, we have
$$G=P^{x_1}  \cup \cdots \cup P^{x_p}\cup A^{y_1}\cup \cdots \cup A^{y_a}\cup  B^{z_1} \cup \cdots \cup B^{z_b}\cup C^{t_1} \cup \cdots \cup C^{t_c}.$$  

Observe that $A$, $B$, and $C$ are centralizers of nonidentity elements, and we have seen that $P$ hass an abelian partition by  centralizers.  Hence, the abelian partition given here consists of centralizers of nonidentity elements of $G$.  Thus, we may use Lemma \ref{minimal} to see that it is a minimal abelian partition. 
A straightforward computation now shows that
$$\vartheta_a (G)=(q-1)(q^2+1)+q^2(q^2+1)/2+q^2(q-1)(q+r+1)/4+q^2(q-1)(q-r+1)/4.$$
After some simplification this leads to
$\vartheta_a (G)=q^4+q^3-q^2+q-1$, as required. 
 \end{proof}


\section{On NAP-groups}

In any group $G$, an element of order $2$ is 
customarily called an {\em involution} and we denote by 
${\rm Inv} (G)$ the set of all involutions of $G$.  
Let $G=G_1\times G_2\times \cdots \times G_n $ be a 
direct product of nontrivial groups $G_i$ of even order. 
Following \cite{FH}, we write 
$${\rm Di}(G):={\rm Inv}(G_1)\times {\rm Inv}(G_2)\times \cdots \times {\rm Inv}(G_n)\subset {\rm Inv}(G),$$ 
and 
$${\rm Dm}(G):=\left(\bigcup_{x\in {\rm Di}(G)}C_G(x)\right)\setminus {\rm Di}(G).$$
The elements of ${\rm Di}(G)$ and  ${\rm Dm}(G)$ are called {\em diagonal involutions} 
and  {\em diagonal mates}, respectively.

Recall that an element $x\in G$ is said to be 
{\em self-centralizing} if $C_G(x)=\langle x\rangle$.
The structure of a finite group containing a 
self-centralizing {\em involution} has been known for a long time. 
In fact, such a group is isomorphic to a 
semidirect product of a non-trivial abelian group 
$A$ of odd order and an involution 
$u$  acts on $A$ by inversion 
(see for instance \cite[Proposition 4.2]{BBW}).  
In other words, $G$ is a Frobenius group 
with abelian kernel $A$ and complement isomorphic to 
${\Bbb Z}_2$. We give here another proof of this result to make the paper
self-contained.

\begin{theorem} 
If $G$ is a group with a
self-centralizing involution, then $G$ is a Frobenius group whose
Frobenius complement has order $2$.
\end{theorem}

\begin{proof}
Let $G$ be a group and let $x$ in $G$ be an involution 
so that $C_G (x)=\langle x\rangle$. 
We first see that a Sylow $2$-subgroup of $G$ has 
order $2$.  To see this, 
let $T$ be a Sylow $2$-subgroup of $G$ containing $x$.  
Observe that $Z(T) \leqslant C_G 
(x) = \langle x \rangle$, so $Z(T) = \langle x \rangle$.  Since $x\in Z(T)$, 
we see $T \leqslant C_G (x) = \langle x \rangle$.  Therefore, $T = \langle x \rangle$ has order $2$.

We may now apply the Burnside normal $p$-complement theorem to see that this implies that $G$ has a normal $2$-complement  (see Satz IV.2.8 of \cite{Hup} for example.).  
Let $N=O_{2'} (G)$.  
We see that $G = 
N\langle x\rangle$.  Since 
$C_G(x) = \langle x\rangle$, 
$x$ acts fixed-point-free on $N$  and so $G$ is a 
Frobenius group with Frobenius kernel $N$ and Frobenius 
complement $\langle x\rangle$.
\end{proof}

In \cite{FH}, it is shown that  a group with a self-centralizing 
involution is a NAP-group. In particular, the dihedral group $D_{2k}$ of order $2k$, with $k\geqslant 3$ odd,  is a NAP-group. This result also derives from a slightly more general result (see Theorem 4.5 in \cite{FH}):  If $G=D_{2k}\times D_{2k}\times \cdots \times D_{2k}$, the direct product $t$ times, where $k$ is odd and $(k+1)^t -2k^t < 0$, then $G$ is a NAP-group.
A more general result of this type is the following theorem. 

\begin{theorem} \label{NAP-Dihedrals} If  
$G=D_{2k_1}\times  D_{2k_2}\times \cdots \times D_{2k_t}$, 
where $k_1, k_2, \ldots, k_t$ are  odd (not necessarily distinct) and 
$$(k_1+1)(k_2+1) \cdots (k_t+1)-2k_1k_2\cdots k_t<0,$$ then $G$ is a NAP-group.
\end{theorem} 
\begin{proof}  Let $G_i=D_{2k_i}$, where $k_i$ is odd $(i=1, 2, \ldots, t)$. 
Suppose  $G$ is isomorphic to 
the direct product $G_1\times G_2\times \cdots\times G_t$.  
First, we claim that $|{\rm Di}(G)|>|{\rm Dm}(G)|$.  
The proof of the claim requires
some calculations. First, we have  
$$|{\rm Di}(G)|=\prod_{i=1}^{t}|{\rm Inv}(G_i)|=k_1k_2\cdots k_t.$$ 
We now compute $|{\rm Dm}(G)|$. Suppose $x=(x_1, x_2, \ldots, x_t)\in {\rm Di}(G)$ 
is a diagonal involution in $G$. 
It is then easy to see that
$$C_G(x)=C_{G_1}(x_1)\times C_{G_2}(x_2)\times \cdots \times C_{G_t}(x_t).$$
In fact, we want to count the elements $g\in C_G(x)\setminus {\rm Di}(G)$,  
when $x$ ranges over ${\rm Di}(G)$, or equivalently, when
$x_i$ ranges over ${\rm Inv}(G_i)$. Now a direct calculation shows that
$$|{\rm Dm}(G)|=(k_1+1)(k_2+1)\cdots (k_t+1)-|{\rm Di}(G)|.$$
Since $(k_1+1)(k_2+1) \cdots (k_t+1)-2k_1k_2\cdots k_t<0$ by 
assumption, it follows $|{\rm Di}(G)|>|{\rm Dm}(G)|$, as claimed.

Next, assume to the contrary that $G$ is an AP-group. 
Since the set of diagonal involutions of $G$ is a noncommuting set, 
we put 
each of these diagonal involutions in a disjoint commuting subset of the abelian partition. 
But then, at least $|{\rm Di}(G)|-|{\rm Dm}(G)|>0$ 
commuting sets will only contain one of these diagonal involutions, 
a contradiction. So $G$ is a NAP-group.
\end{proof}

As the following corollary shows, we can find a NAP-group $G$ for which $|C_G (x)|$ is ``large'' for every element $x \in G$.

\begin{corollary} For any integer $m\geqslant 1$, there exist 
NAP-groups $G$ satisfying $|C_G(x)|>m$, for every $x$ in $G$.
\end{corollary} 

\begin{proof}  Suppose $t$ is a natural number such that $2^t>m$ and 
consider the group $G$ as in Theorem \ref{NAP-Dihedrals}. \end{proof}

                           
\subsection{Wreath products} 
Before stating our results in this subsection, we need to introduce some
additional notation. Let $K$ and $H$ be groups and suppose $H$ is a permutation group of degree $n$, that is $H \leqslant {\Bbb S}_n$.
Then, we will denote the wreath product of $K$ by $H$ by $K \wr  H$, where $H$ acts on the base group $$B(K, n):=\underbrace{K\times K\times \cdots \times K}_{n-{\rm times}},$$ by permuting the components. In the case that $K$ has even order, we will call a diagonal involution in  $B(K, n)$   a {\em constant diagonal involution} if it has the form $(t, t, \ldots, t)$, where $t\in {\rm Inv}(K)$.  Let ${\rm Di}_{\rm nc}(K \wr  H)$ denote the set of nonconstant diagonal involutions, and ${\rm Dm}_{\rm nc}(K \wr  H)$ denote the set of mates to the  nonconstant diagonal involutions.

In what follows, we shall focus our attention on the case $K=D_{2k}$ and $H\leqslant {\Bbb S}_n$  a permutation group of degree $n$, and consider
the wreath product $G=D_{2k}\wr H$.
Note that  $B(D_{2k}, n)$  has $k$ constant diagonal involutions.

\begin{theorem}\label{wreath}  Suppose that $G=D_{2k}\wr   \langle (12\dots p)\rangle$, where $k$ is an odd positive integer and $p$ an odd prime, which satisfy $(k+1)^p+k-2k^p<0$. Then $G$ has no abelian partition.
\end{theorem}

\begin{proof} 
It is clear that $|{\rm Di}_{\rm nc}(G)|=k^p-k$. Also,  a Sylow $2$-subgroup $P$ of $G$ is an elementary abelian and  any nonconstant diagonal involution $t$ in $P$ commutes  only with elements of $P$. Thus ${\rm Dm}_{\rm nc}(G)$ consist of the identity and all the  nondiagonal involutions. Now a direct computation shows that
$$|{\rm Dm}_{\rm nc}(G)|={p\choose0}+{p\choose1}k+{p\choose2}k^2+\dots+{p\choose p-1}k^{p-1}=(k+1)^p-k^p,$$
and the assumption $(k+1)^p+k-2k^p<0$ implies that $|{\rm Dm}_{\rm nc}(G)|< |{\rm Di}_{\rm nc}(G)|$.  If we put each of the nonconstant diagonal involutions in a disjoint commuting subset of $G$, then at least $2k^p-(k+1)^p-k>0$ singleton sets  will be needed. This shows that $G$ has no abelian partition. \end{proof}

Note that Theorem \ref{wreath} gives us many examples of NAP-groups that are not direct products of groups with self-centralizing involutions, and that the involutions do not generate the group.

 We will call a diagonal involution in  $B(D_{2k}, n)$, where $k>n$, a {\em fixed point free diagonal involution} if it has the form $(t_1, t_2, \dots, t_n)$, where $t_i\neq t_j$ whenever $i\neq j$. Let $G=D_{2k}\wr H$, where $H\leqslant {\Bbb S}_n$  is a permutation group of degree $n$. Let ${\rm Di}_{\rm fp}(G)$ denote the set of fixed point free diagonal involutions of $G$, and  ${\rm Dm}_{\rm fp}(G)$ denote the set of mates to the  fixed point free diagonal involutions. Note that  $G$  has $k!/n!$ fixed point free diagonal involutions and that $|{\rm Dm}_{\rm fp}(G)|\leqslant |{\rm Dm}(G)|$, where ${\rm Dm}(G)$ is the set of diagonal mates in  $B(D_{2k}, n)$. 

\begin{lm}\label{lemma-FPF} For any positive integer $n$, there exists an integer $k>n$ such that $$\frac{k!}{n!}>(1+k)^n-k^n.$$\end{lm}
The proof is immediate because  $$\lim_{k\to\infty}\frac{k!}{n!((1+k)^n-k^n)}=\infty. $$ 
Given an integer $n>0$, we will denote by $\gamma{(n)}$ the {\em smallest} positive integer such that $$\frac{\gamma{(n)}!}{n!}>(1+\gamma{(n)})^n-\gamma{(n)}^n.$$  

We now obtain Theorem A from the Introduction.

\begin{theorem}\label{theorem-FPF}  The following conditions hold:
\begin{itemize}
\item[{\rm (1)}] Every finite group is isomorphic to a  subgroup of an AP-group. 
\item[{\rm (2)}] Every finite group is isomorphic to a  subgroup of a NAP-group. 
\end{itemize}
 \end{theorem}
\begin{proof} Let $H$ be a finite group of order $h$ and $A$ be an abelian group. Part (1) follows from the fact that  $H\times A$  is an AP-group (see Theorem 2.16 in \cite{FH}). To prove (2), regarding $H$ as a permutation group 
of degree $h$ (Cayley's Theorem), we can consider the wreath product  $G=D_{2 \gamma{(h)}} \wr H$. Now, we have $${\rm Di}_{\rm fp}(G)=\frac{\gamma{(h)}!}{h!}>(1+\gamma{(h)})^{h}-\gamma{(h)}^{h}=|{\rm Dm}(G)|\geqslant |{\rm Dm}_{\rm fp}(G)|,$$
 which shows that $G$ is a NAP-group. Part (2) is now immediate
\end{proof}

\section{Conclusions and Future Directions}

In this paper, we have shown that both the class of AP-groups and NAP-groups is large. We have also found bounds for the AP-degree of groups and shown cases where those bounds are sharp. This builds on previous work where a classification of groups with AP-degrees of $1$, $2$, $3$, and $4$ was obtained, along with work classifying several families of groups which were (N)AP-groups.  However, the complete classification of AP-groups remains a mystery. 

As stated in the introduction, an interesting question is whether or not all simple groups are AP-groups. To rephrase: are their any NAP-groups that are simple? Also, while we have determined the minimal AP-degree for AC-groups and, conditional on understanding the AP-degree of certain subgroups, we have determined the AP-degree of Frobenius groups, the problem of determining the AP-degree for other families of groups is wide open. This, of course, can be approached from a different perspective: asking, up to isomorphism, for a classification of groups according to AP-degree. To our knowledge, this problem remains untouched for degree $5$ and higher.

\noindent  {\sc  T. Foguel}\\[0.2cm]
{\sc Department of Mathematics and Computer Science, Adelphi University,}\\
{\sc Garden City, NY $11010$, United States of
America}\\[0.1cm]
{\em E-mail address}: {\tt tfoguel@adelphi.edu}\\[0.3cm]
{\sc  J. Hiller}\\[0.2cm]
{\sc Department of Mathematics and Computer Science, Adelphi University,}\\ 
{\sc Garden City, NY $11010$,  United States of
America}\\[0.1cm]
{\em E-mail address}: {\tt johiller@adelphi.edu}\\[0.3cm]
{\sc Mark L. Lewis}\\[0.2cm]
{\sc Department of Mathematical Sciences, Kent State
University,}\\ {\sc  Kent, Ohio $44242$, United States of
America}\\[0.1cm]
{\em E-mail address}: {\tt  lewis@math.kent.edu}\\[0.3cm]
 {\sc  A. R. Moghaddamfar}\\[0.2cm]
{\sc Faculty of Mathematics, K. N. Toosi
University of Technology,
 P. O. Box $16765$--$3381$, Tehran, Iran}\\[0.1cm]
{\em E-mail addresses}:  {\tt
moghadam@kntu.ac.ir}, and {\tt moghadam@ipm.ir}
\end{document}